\documentclass[12pt,fullpage]{amsart}

\usepackage{amssymb,amsmath,amsthm,amsfonts}
\usepackage[mathscr]{euscript}
\usepackage{dsfont}
\usepackage{mathrsfs}
\usepackage{esint}
\usepackage{graphicx}
\usepackage{enumitem}
\usepackage{geometry}
\renewcommand{\leq}{\leqslant}

\renewcommand{\geq}{\geqslant}

\usepackage{color}
\definecolor{citation}{rgb}{0.2,0.5,0.2}
\definecolor{formula}{rgb}{0.1,0.2,0.5}
\definecolor{url}{rgb}{0,0.2,0.7}
\usepackage[colorlinks=true,linkcolor=formula,urlcolor=url,citecolor=citation]{hyperref}

\newtheoremstyle%
   {olivier}
   {1.5ex}
   {1.5ex}
   {\sl}
   {11.5pt}
   {\bf\sc}   
   {.~---}
   {1ex}
   {}

\newtheoremstyle%
   {oliv0}
   {1.5ex}
   {1.5ex}
   {\sl}
   {11.5pt}
   {\bf\sc}   
   {~---}
   {1ex}
   {}

\newtheoremstyle%
   {clovis}
   {1.5ex}
   {1.5ex}
   {\normalfont}
   {11.5pt}
   {\sl}               
   {.~---}
   {1ex}
   {}

\newtheorem{theo}{Theorem}[section]

\newtheorem{lemma}[theo]{Lemma}
\newtheorem{prop}[theo]{Proposition}

\theoremstyle{definition}

\theoremstyle{fact}

\theoremstyle{remark}
\newtheorem{remark}[theo]{Remark}

\numberwithin{equation}{section}

\def\R {\mathbb{R}}

\def\N {\mathbb{N}}
\def\S {\mathbb{S}}

\def\eps{\varepsilon}

\usepackage{scalerel}[2014/03/10]
\usepackage[usestackEOL]{stackengine}
\def\dashint{\,\ThisStyle{\ensurestackMath{%
  \stackinset{c}{.2\LMpt}{c}{.5\LMpt}{\SavedStyle-}{\SavedStyle\phantom{\int}}}%
  \setbox0=\hbox{$\SavedStyle\int\,$}\kern-\wd0}\int}
\def\ddashint{\,\ThisStyle{\ensurestackMath{%
  \stackinset{c}{.2\LMpt}{c}{.5\LMpt+.2\LMex}{\SavedStyle-}{%
    \stackinset{c}{.2\LMpt}{c}{.5\LMpt-.2\LMex}{\SavedStyle-}{%
      \SavedStyle\phantom{\int}}}}\setbox0=\hbox{$\SavedStyle\int\,$}\kern-\wd0}\int}

\newcommand{\baa}{\begin{array}}
\newcommand{\eaa}{\end{array}}
\newcommand{\ba}{\begin{eqnarray}}
\newcommand{\ea}{\end{eqnarray}}
\newcommand{\be}{\begin{equation}}
\newcommand{\ee}{\end{equation}}

\newlength{\defbaselineskip}
\newcommand{\setlinespacing}[1]
           {\setlength{\baselineskip}{#1 \defbaselineskip}}

\allowdisplaybreaks[4]

\begin{document}

\title[On the role of the range of dispersal]{On the role of the range of dispersal in a nonlocal Fisher-KPP equation: an asymptotic analysis}

\author[Julien Brasseur]{Julien Brasseur}
\address{\'Ecole des Hautes \'Etudes en Sciences Sociales, PSL Research University, CNRS, Centre d'Analyse et de Math\'ematique Sociales, Paris, France}
\email{julien.brasseur@ehess.fr, julienbrasseur@wanadoo.fr}

\begin{abstract}
In this paper, we study the asymptotic behavior as $\varepsilon\to0^+$ of solutions $u_\varepsilon$ to the nonlocal stationary Fisher-KPP type equation
$$ \frac{1}{\varepsilon^m}\int_{\mathbb{R}^N}J_\varepsilon(x-y)(u_\varepsilon(y)-u_\varepsilon(x))\hspace{0.1em}\mathrm{d}y+u_\varepsilon(x)(a(x)-u_\varepsilon(x))=0\text{ in }\mathbb{R}^N, $$
where $\varepsilon>0$ and $0\leq m<2$. Under rather mild assumptions and using very little technology,
we prove that there exists one and only one positive solution $u_\varepsilon$
and that $u_\varepsilon\to a^+$ as $\varepsilon\to0^+$ where $a^+=\max\{0,a\}$. This
generalizes the previously known results and answers an open question
raised by Berestycki, Coville and Vo.
Our method of proof is also of independent interest as it shows how to reduce this nonlocal problem to a local one.
The sharpness of our assumptions is also briefly discussed.
\end{abstract}

\subjclass[2010]{35J60}


\maketitle

\tableofcontents

\section{Introduction}

\subsection{Biological context}
This paper is motivated by the study of persistence criteria for populations with long range dispersal strategies.
Long range dispersal is a frequently observed feature in ecology.
It arises in many situations ranging from plant fecundity \cite{Cain,Clark,Clark2,Schurr} to movement patterns of marine predators \cite{Bartumeus,Bartumeus2,Hallat,Humphries}.
In this context, the evolution of the density of population, $u(x,t)$, is commonly described
by a nonlocal reaction-diffusion equation of the form
\begin{align}
\frac{\partial u}{\partial t}(x,t)=D\int_{\R^N}J(x-y)(u(y,t)-u(x,t))\hspace{0.1em}\mathrm{d}y+f(x,u(x,t)), \label{EQ:NON}
\end{align}
for $(x,t)\in\R^N\times[0,\infty)$; where $D$ is a dispersal rate (or diffusion coefficient), $J$ is a dispersal kernel modelling the probability
to ``jump" from one location to another, and $f$ is a nonlinear reaction term that accounts for
the demographic variations of the population.


In this paper, we will consider populations that have a \emph{bounded ecological niche}, i.e. populations which cannot survive outside a bounded set. Precisely, we shall consider KPP-type nonlinearities of the form
$$ f(x,s)=s\left(a(x)-s\right) \text{ with } (x,s)\in\R^N\times\R, $$
where $a\in C(\R^N)\cap L^\infty(\R^N)$ is a function (that can be thought of as modelling the available resource) which is such that $a^+=\max\{0,a\}\not\equiv0$ and
\begin{align}
\begin{array}{l}
\displaystyle\limsup_{|x|\to\infty}\hspace{0.1em}a(x)<0.
\end{array} \label{HYP:A1}
\end{align}

From the mathematical point of view, studying the persistence or the extinction of a given species amounts to establishing the existence or nonexistence of positive solutions to an equation of the type of \eqref{EQ:NON} (whose precise form may vary depending on the type of behavior one wishes to describe). This type of nonlocal problem is currently receiving a lot of attention and has been studied under various perspectives, see \cite{Bates,BaZh,CovDup07,Cov07,Kao,Yagisita}. Here, we follow the approach initiated by Hutson \emph{et al.} in \cite{Hutson}. Namely, we consider the equation
\begin{align}
\frac{\partial u_\eps}{\partial t}(x,t)=\frac{1}{\eps^m}\int_{\R^N}J_\eps(x-y)(u_\eps(y,t)-u_\eps(x,t))\hspace{0.1em}\mathrm{d}y+u_\eps(x,t)(a(x)-u_\eps(x,t)), \label{Hutsonn}
\end{align}
for $(x,t)\in\R^N\times[0,\infty)$; where $0\leq m\leq 2$ is a ``cost parameter",
\begin{align}
J_\eps(z)=\frac{1}{\eps^N}\hspace{0.1em}J\left(\frac{z}{\eps}\right)\text{ for some }J\in L^1(\R^N), \label{Disp:Ker}
\end{align}
and $\eps>0$ is a measure of the range of dispersal. 

Before going any further, let us say a brief word about the meaning of \eqref{Hutsonn}. The key idea behind this model relies on the notion of \emph{dispersal budget} (introduced in \cite{Hutson}). In a nutshell, it consists in assuming that \emph{the amount of energy per individual that the species can use to disperse is fixed} (because of the environmental and developmental constraints) and that \emph{the displacement of the individuals has a cost} (reflecting the amount of energy required to disperse) which, for simplicity, is assumed to be proportional to $c(x)=|x|^m$ with $0\leq m\leq 2$.
If the dispersal kernel depends on the range of dispersal as in \eqref{Disp:Ker}, it can then be shown (see e.g. \cite{BCV,Hutson}) that the rate of dispersal is of the form $D_\eps\sim 1/\eps^m$, which thereby yields equation \eqref{Hutsonn}.
In other words, \emph{there is a trade-off between the number of offspring and the average dispersal distance among offspring}.
To illustrate this, consider, for instance, a population of trees that produces and disperses seeds. The population may then ``choose" between two opposite strategies: either it disperses few seeds ($1/\eps^m\ll1$) over large distances ($\eps\gg1$), or it disperses many of them ($1/\eps^m\gg1$) over short distances ($\eps\ll1$).\par

Our main concern in this paper is to understand the influence of the cost parameter on the asymptotic properties of the solutions.
That is, we wish to address the following question:
\begin{center}
\emph{How does the cost of displacement impact the persistence} \\
\emph{strategies of a given population?}
\end{center}
\smallskip

To the best of our knowledge, the first extensive study of this problem goes back to \cite{BCV} (see also \cite{LI,VoShen2,Vo}), where the assumption \eqref{HYP:A1} was also considered.
There, it was shown that positive stationary solutions can be seen as the outcome of an invasion process (see \cite[Theorem 1.1]{BCV}), i.e. that their existence/nonexistence gives the right persistence criteria. For this reason, we will focus on the corresponding stationary equation, namely
\begin{align}
\frac{1}{\eps^m}(J_\eps\ast u_\eps-u_\eps)+u_\eps(a-u_\eps)=0 \text{ in }\R^N. \label{KPP}
\end{align}

Equation \eqref{KPP} was first studied in detail by Berestycki, Coville and Vo in \cite{BCV}.
They proved that, for large values of $\eps$, persistence always occurs when $0<m\leq 2$ and they obtain the precise asymptotics of the solution when $\eps\to\infty$. When $m=0$, they show that if the resource is ``too small" (i.e. if $\sup_{\R^N}a<1$) then the population either dies out above some threshold $\eps_0>0$ or vanishes asymptotically as $\eps\to\infty$.

However, \emph{if we have a rather clear picture when $\eps\gg1$ it is not quite the case when $\eps\ll1$}, except in the particular case $m=2$. In this case, Berestycki \emph{et al.} \cite[Theorem 1.4]{BCV} show that, for small values of $\eps$, a nontrivial solution to \eqref{KPP} exists if, and only if, the first eigenvalue of some linear elliptic operator is negative and, in this case, they determine the precise asymptotics as $\eps\to0^+$.

The case $0\leq m<2$, however, turns out to be more involved. Berestycki \emph{et al.} proved (see \cite[Theorems 1.2-1.3]{BCV}) that if $0<\eps<\eps_1$ for some $\eps_1>0$, then there always exists a unique bounded, continuous, positive solution $u_\eps$ to \eqref{KPP}. However, its precise behavior as $\eps\to0^+$ is still an open problem. The best known result in this direction states that $u_\eps$ converges weakly towards some nonnegative function $v\in L^\infty(\R^N)$ solving
\begin{align*}
v(x)(a(x)-v(x))=0\text{ for all }x\in\R^N. 
\end{align*}
Unfortunately, the above equation admits infinitely many solutions, so it may happen that $v\equiv0$ (extinction) or that $v=a^+\mathds{1}_\omega$ for some $\omega\subset\mathrm{supp}(a^+)$ (persistence in a given area of the ecological niche). The goal of this paper is to complete this picture by showing that $v=a^+$ is the only possible solution, thus enforcing that \emph{short range dispersal strategies subject to sub-quadratic costs always yield persistence}.

\subsection{Assumptions and main results}
Throughout the paper, we shall assume that
\begin{align}
\left\{\begin{array}{l}
J\in L^1(\R^N)\text{ is nonnegative, radially symmetric,} \\
\text{with unit mass and finite $m$-th order moment}.
\end{array}\right. \label{HYP:Jweak}
\end{align}
For the convenience of the reader, we recall that the space $\mathring{B}_{p,\infty}^s(\R^N)$ with $s\in(0,2)$ and $1\leq p<\infty$, is the closure of $C_c^\infty(\R^N)$ in the Besov space $B_{p,\infty}^s(\R^N)$ or, equivalently,
$$ \mathring{B}_{p,\infty}^s(\R^N)=\left\{f\in L^p(\R^N)\ ;\ \lim_{|h|\to0}\hspace{0.1em}\frac{\|f(\cdot+h)-2f+f(\cdot-h)\|_{L^p(\R^N)}}{|h|^s}=0\right\}, $$
see e.g. \cite[Proposition 5.6]{JB}. By convention, we set $\mathring{B}_{p,\infty}^0(\R^N)=L^p(\R^N)$.
\smallskip

Our first main result states that short range dispersal strategies always yield persistence when $0\leq m<2$, thus answering an open question raised in \cite{BCV}. 

\begin{theo}\label{THEO}
Let $0\leq m<2$. Assume \eqref{HYP:A1} and \eqref{HYP:Jweak}. 
Then, there exists $\eps_0>0$ such that, for all $0<\eps<\eps_0$, \eqref{KPP} admits a unique positive minimal solution $u_\eps\in C_0(\R^N)$.
Moreover,
$$ \liminf_{\eps\to0^+}\hspace{0.1em}u_\eps(x)\geq a^+(x)\text{ for all }x\in\R^N. $$
\end{theo}
\begin{remark}
While we believe that our assumptions are close to being sharp, the existence of a counterexample is still an open question.
Nevertheless, we can prove that, \emph{if $0<m<2$ and if $J$ has an infinite $\beta$-th order moment for some $0<\beta<m$, then \eqref{KPP} does not admit uniform (with respect to $\eps$) sub-solutions controlled by $a^+$} (see Proposition~\ref{MOMENT}). This suggests that, if $J$ is too heavily tailed (i.e. if its $m$-th order moment is infinite), then the conclusion of Theorem~\ref{THEO} fails (i.e. either \eqref{KPP} admits no positive solution for $\eps$ small or the positive minimal solution vanishes asymptotically as $\eps\to0^+$). This would be consistent since $J$ is nothing but a probability density, so the heavier the tail, the more likely the individuals are to favor long range jumps, which seems incompatible with the rate of dispersal $1/\eps^m$ becoming arbitrarily large as the range of dispersal $\eps$ becomes arbitrarily small.
It is worth mentioning that, in similar contexts, fat-tailed kernels are known to induce a dramatically different behavior of the solutions, see e.g. \cite{Bouin,Garnier}.
It would be of interest to investigate this question further, as fat-tailed kernels are known to better account for the dispersal of individuals in various contexts (it is observed, for example, in river fishes \cite{Radinger}, in the tansy beetle \emph{Chrysolina graminis} \cite{Chapman} or in rapid plant migration \cite{Clark,Clark2}, see also \cite{Allen,Petrovskii}).
\end{remark}
%

It would be desirable to have the existence of a unique positive solution (rather than of a unique positive minimal solution) as well as more precise asymptotics. However, establishing this is quite delicate if the minimal solution $u_\eps$ is not known to be integrable. On the other hand, it can be shown that the decay of $u_\eps$ is intimately related to the tail of the kernel: roughly speaking, if $J$ has a finite $\beta$-th order moment, then the unique positive minimal solution decays as $|x|^{-\beta}$. Using this dichotomy, we show that it is possible to sharpen the conclusion of Theorem~\ref{THEO} up to a slight additional assumption on $J$. Precisely,

\begin{theo}\label{THEO02}
Let $0\leq m<2$. Assume \eqref{HYP:A1} and \eqref{HYP:Jweak}. Suppose, in addition, that $J$ has a finite $\beta$-th order moment for some $\beta>N$.
Then, there exists $\eps_0>0$ such that, for all $0<\eps<\eps_0$, \eqref{KPP} admits a unique bounded positive solution $u_\eps\in L^1\cap C_0(\R^N)$.
Moreover, if $a^+\in \mathring{B}_{1,\infty}^{m}(\R^N)$, then the solution $u_\eps$ to \eqref{KPP} converges almost everywhere to $a^+$ as $\eps\to0^+$.
\end{theo}
\begin{remark}
We do not know whether it is possible to get rid of the assumption that $J$ has a finite $\beta$-th order moment for some $\beta>N$.
Note that if $N=1$ and $1<m<2$, then this last assumption is not needed, which might indicate that it is not necessary. However, the regularity assumption on $a^+$ emerges very naturally in proof, suggesting that it is sharp.
\end{remark}

Combining Theorem~\ref{THEO02} with \cite[Theorem 1.3]{BCV}, we obtain that if $0<m<2$, then $u_\eps\to a^+$ (at least pointwise) when both $\eps\to0^+$ and $\eps\to\infty$. This means that \emph{for both strategies $\eps\ll1$ and $\eps\gg1$ the population will tend to match the resource}, thus yielding persistence in any case. This highly contrasts with the cases $m=0$ and $m=2$. In the latter case, although we still have that $u_\eps\to a^+$ when $\eps\to\infty$, equation \eqref{KPP} may not even have positive solutions at all when $\eps$ is small, depending on the sign of the first eigenvalue of $(2N)^{-1}M_2(J)\hspace{0.1em}\Delta+a$ (which is merely a function of the resource $a$), where
$$ M_2(J)=\int_{\R^N}J(x)|x|^2\mathrm{d}x. $$
Otherwise said, depending on the precise form of $a$, the population may go extinct when the spread of dispersal $\eps$ is too small. When $m=0$, the situation is somehow ``opposite" to the case $m=2$, as persistence may not occur for long range dispersal strategies. This indicates that \emph{the persistence strategy of the population strongly depends on what it costs for the individuals to disperse}.


\subsection{Strategy of proof}\label{SE:strategy}
The main difficulty in establishing Theorems~\ref{THEO}-\ref{THEO02} lies in the lack of compactness of \eqref{KPP}. It is known that solutions to \eqref{KPP} satisfy
\begin{align}
\int_{\R^N}\int_{\R^N}\rho_\eps(x-y)\frac{|u_\eps(x)-u_\eps(y)|^2}{|x-y|^m}\mathrm{d}x\mathrm{d}y\leq C\text{ as }\eps\to0^+, \label{BBM}
\end{align}
where $\rho_\eps(z)=\eps^{-m}|z|^mJ_\eps(z)$, see \cite[Lemma 5.1(ii)]{BCV}. This inequality is in fact the key tool which allowed Berestycki \emph{et al.} to handle the case $m=2$. Indeed, \eqref{BBM} together with the recent characterisation of Sobolev spaces derived by Bourgain, Brezis, Mironescu \cite{BBM} and Ponce \cite{Ponce} implies that $(u_\eps)_{\eps>0}$ is relatively compact in $L_{\mathrm{loc}}^2(\R^N)$ and that it converges along a subsequence to some function $v\in L^2(\R^N)$ satisfying
$$ \lim_{\eps\to0^+}\int_{\R^N}\int_{\R^N}\rho_\eps(x-y)\frac{|v(x)-v(y)|^2}{|x-y|^2}\mathrm{d}x\mathrm{d}y\leq C, $$
which, by a result of Bourgain \emph{et al.} \cite[Theorem 2]{BBM}, is equivalent to saying that $v$ belongs to the Sobolev space $H^1(\R^N)$. Then, relying on standard elliptic theory it can be shown that $v$ is the unique nontrivial solution to
$$ \frac{M_2(J)}{2N}\hspace{0.1em}\Delta v+v(a-v)=0\text{ in }\R^N. $$
However, we have shown in \cite{JB} that, although the functional arising in \eqref{BBM} provides a characterisation of a fractional version of $H^1(\R^N)$ when $0<m<2$ (see \cite[Theorem 2.3]{JB}), \emph{it does not yield precompactness in $L_{\mathrm{loc}}^2(\R^N)$} (see \cite[Theorem 2.15]{JB}).

To overcome this lack of compactness, we need to rely on an entirely different approach.
The heart of our strategy is based on the construction of an appropriate sequence of sub-solutions.
The main idea behind our construction stems from the observation that if $\Omega\subset\R^N$ is a domain and $J$ is radial and compactly supported, then any function $\varphi\in C^2(\Omega)$ that is subharmonic in $\Omega$ (i.e. $\Delta\varphi\geq0$ in $\Omega$) satisfies the ``generalized mean value inequality":
$$ \varphi(x)\leq \int_{\R^N}J_\eps(x-y)\hspace{0.1em}\varphi(y)\hspace{0.1em}\mathrm{d}y, $$
for all $x\in\Omega$ and $\eps>0$ small enough. (The usual mean value inequality corresponds to the case where $J(z)=|B_1|^{-1}\mathds{1}_{B_1}(z)$.) Based on this observation, we prove that it is possible to construct
continuous global sub-solutions to \eqref{KPP} by considering functions which are subharmonic outside some ball around the origin and ``well-behaved" inside it.
In this way, we are able to \emph{reduce this nonlocal problem to a local one}.
Then, relying on Markov's inequality we show that this procedure works as well for general kernels having finite $m$-th order moment
(which, as discussed in the Appendix at the end of the paper, is close to being a sharp requirement for a ``good" sub-solution to exist).

This approach presents a considerable advantage since it yields \emph{simultaneously} existence, uniqueness
and asymptotic results \emph{without relying on the spectral theory for convolution operators}.
In particular, not only do we obtain an alternative proof of the existence and uniqueness results of \cite{BCV} using very little technology, but we even extend it to kernels which may be fat-tailed, i.e. to the class of kernels $J$ satisfying \eqref{HYP:Jweak}.


\subsection{Notations}\label{SE:NTN}

Let us list a few notations that will be used throughout the paper.

As usual, $\S^{N-1}$ denotes the unit sphere of $\R^N$ and $B_R(x)$ the open Euclidean ball of radius $R>0$ centred at $x\in\R^N$ (when $x=0$, we simply write $B_R$). The $N$-dimensional Hausdorff measure will be denoted by $\mathcal{H}^{N}$.
Given a function $f:\R^N\to\R$, we denote by $f^+$ (resp. $f^-$) its positive part (resp. negative part) given by $f^+=\max\{0,f\}$ (resp. $f^-=\max\{0,-f\}$).
The set-theoretic support of $f$ will be denoted by
$$ \mathrm{supp}(f):=\left\{x\in\R^N; f(x)\neq0\right\}, $$
and its essential support (that is, the complement of the union of all open sets on which $f$ vanishes almost everywhere) will be denoted by $\mathrm{ess}\,\mathrm{supp}(f)$.
Observe that, within these definitions, if $f$ is continuous then its essential support coincides with the \emph{closure} of its set-theoretic support.
Furthermore, for a measurable set $\Omega\subset\R^N$, we denote by $|\Omega|$ its Lebesgue measure and by $\mathds{1}_\Omega$ its characteristic function.
For $f\in L^1(\R^N)$ and $\beta\geq0$, we denote by
$$ M_{\beta}(f):=\int_{\R^N}f(x)|x|^{\beta}\mathrm{d}x, $$
the $\beta$-th order moment of $f$.
Lastly, we will denote by $C_0(\R^N)$ (resp. $C_c(\R^N)$) the space of continuous functions that vanish at infinity (resp. with compact support).


\section{Preliminaries}

In this section, we list some general qualitative/a priori results for the solutions to \eqref{KPP} (if any) which will be useful for later purposes.
\begin{lemma}[Strong maximum principle]\label{LE:SIGNE}
Let $0\leq m\leq2$. Assume that $a\in C(\R^N)$ and that $J\in L^1(\R^N)$ is a nonnegative radial kernel with unit mass. Suppose that \eqref{KPP} admits a nonnegative solution $u_\eps\in L^\infty(\R^N)$. Then, either $u_\eps>0$ a.e. in $\R^N$ or $u_\eps\equiv0$ a.e. in $\R^N$.
\end{lemma}
\begin{remark}
For related results in the nonlocal framework, the reader may consult \cite{Cov08}.
\end{remark}
\begin{proof}
Let $u_\eps\in L^\infty(\R^N)$ be a nonnegative solution to \eqref{KPP} and suppose that $u_\eps$ admits a Lebesgue point $x_0\in\R^N$ such that $u_\eps(x_0)=0$. Using the equation satisfied by $u_\eps$ it follows that $u_\eps(y)=0$ for a.e. $y\in x_0+\mathrm{ess}\,\mathrm{supp}(J_\eps)$. By iteration, we find that
\begin{align}
u_\eps(y)=0 \text{ for a.e. }y\in x_0+\mathrm{ess}\,\mathrm{supp}(J_\eps)+\mathrm{ess}\,\mathrm{supp}(J_\eps). \label{iter:2:eps}
\end{align}
Let $R>0$ be such that $\Lambda_\eps:=\mathrm{ess}\,\mathrm{supp}(J_\eps)\cap \overline{B_{R}}$ has positive Lebesgue measure.
Since the function $G_\eps$ given by $G_\eps(x):=\mathds{1}_{\hspace{0.1em}\Lambda_\eps}\ast\mathds{1}_{\hspace{0.1em}\Lambda_\eps}(x)$ is continuous and since, on the other hand, $G_\eps(0)=|\Lambda_\eps|>0$, we deduce that there is some $\delta>0$ such that
$$ B_\delta\subset\mathrm{ess}\,\mathrm{supp}(G_\eps)\subset \Lambda_\eps+\Lambda_\eps\subset\mathrm{ess}\,\mathrm{supp}(J_\eps)+\mathrm{ess}\,\mathrm{supp}(J_\eps). $$
Plugging this in \eqref{iter:2:eps}, we obtain that $u_\eps(y)=0 \text{ for a.e. }y\in B_\delta(x_0)$.
By induction, we find that $u_\eps(y)=0$ for a.e. $y\in B_{\delta k}(x_0)$ and all $k\in\N\setminus\{0\}$. Hence, $u_\eps(y)=0$ for a.e. $y\in\R^N$.
This enforces that either $u_\eps>0$ a.e. in $\R^N$ or $u_\eps\equiv0$ a.e. in $\R^N$.
\end{proof}
We will also need the following lower bound when $m=0$.
\begin{lemma}\label{LE:LOWER}
Let $m=0$. Assume that $a\in C(\R^N)$ and that $J\in L^1(\R^N)$ is a nonnegative radial kernel with unit mass. Suppose that \eqref{KPP} admits a nontrivial nonnegative solution $u_\eps\in L^\infty(\R^N)$. Then, the following estimate holds
\begin{align*}
u_\eps(x)>(a(x)-1)^+\text{ for a.e. }x\in\R^N.
\end{align*}
\end{lemma}
\begin{proof}
Suppose, by contradiction, that $u_\eps$ admits a Lebesgue point $x_0\in\R^N$ such that $a(x_0)\geq 1+u_\eps(x_0)$. Then, it follows from the equation satisfied by $u_\eps$ that
$$ 0=J_\eps\ast u_\eps(x_0)-u_\eps(x_0)+u_\eps(x_0)(a(x_0)-u_\eps(x_0))\geq J_\eps\ast u_\eps(x_0)\geq0. $$
Thus, $u_\eps(y)=0$ for a.e. $y\in x_0+\mathrm{ess}\,\mathrm{supp}(J_\eps)$. By Lemma~\ref{LE:SIGNE}, we find that $u_\eps(y)=0$ for a.e. $y\in\R^N$, which contradicts the fact that $u_\eps$ is nontrivial. Hence $u_\eps(x)>(a(x)-1)$ for a.e. $x\in\R^N$. The conclusion now follows from the fact that $u_\eps>0$ a.e. in $\R^N$. 
\end{proof}

Lastly, we state a regularity result for solutions to \eqref{KPP} which will play an important role in the sequel.

\begin{prop}\label{LE:REG:POS}
Let $0\leq m\leq2$. Assume \eqref{HYP:A1} and that $J\in L^1(\R^N)$ is a nonnegative radial kernel with unit mass.
Suppose that \eqref{KPP} admits a nonnegative solution $u_\eps\in L^\infty(\R^N)$ for all $0<\eps<\eps_0$ and some $\eps_0>0$.
Then, there exists $\eps_{m,a^+}>0$ such that, for all $0<\eps<\min\{\eps_0,\eps_{m,a^+}\}$, we have $u_\eps\in C_0(\R^N)$.
Moreover, either $u_\eps>0$ or $u_\eps\equiv0$ in $\R^N$.
\end{prop}
\begin{remark}
We mention that related results in similar contexts may be found in \cite{BaZh,Cantrell}, although our arguments here are more direct as they do not require any incursion through Krein-Rutman theory nor do they require any \emph{reductio ad absurdum} argument.
\end{remark}
\begin{proof}
We already know (from Lemma~\ref{LE:SIGNE}) that either $u_\eps\equiv0$ or $u_\eps>0$ a.e. in $\R^N$. Hence, it suffices to prove that if $u_\eps$ is nontrivial, then it admits a representative in its class of equivalence that belongs to $C_0(\R^N)$. So let $u_\eps\in L^\infty(\R^N)$ be a nonnegative nontrivial solution to \eqref{KPP}. Let us first prove that $u_\eps$ can be redefined up to a negligible set as a continuous function in $\R^N$. For it, let us set $H_\eps(x,s):=s(1-\eps^m[a(x)-s])$. Since $u_\eps$ solves \eqref{KPP} a.e. in $\R^N$, there is then a null set $\mathscr{N}\subset\R^N$ such that, for all $x,y\in\R^N\setminus\mathscr{N}$,
\begin{align}
H_\eps(x,u_\eps(x))-H_\eps(y,u_\eps(y))=\int_{\R^N}\left[J_\eps(x-z)-J_\eps(y-z)\right]u_\eps(z)\hspace{0.1em}\mathrm{d}z=:\mathcal{J}_\eps(x,y). \label{gxy2}
\end{align}
Since $J_\eps\in L^1(\R^N)$, $u_\eps\in L^\infty(\R^N)$, the function $\mathcal{J}_\eps$ defined by the right-hand side of \eqref{gxy2} can actually be defined in $\R^N\times\R^N$ and it is uniformly continuous in $\R^N\times\R^N$ (due to the continuity of translations in $L^1$).

Let us first consider the case $m=0$. As will be explained later on, the case $0<m<2$ is actually simpler and will follow from the same type of arguments. Now, since
$$ \partial_sH_\eps(x,s)=1-\eps^m[a(x)-2s]=1+2s-a(x)>0\text{ for all } s\in\left(\frac{1}{2}(a(x)-1)^+,\infty\right), $$
since $H_\eps(x,\frac{1}{2}(a(x)-1)^+)=-\frac{1}{4}[(a(x)-1)^+]^2$ and $H_\eps(x,\infty)=\infty$ for all $x\in\R^N$,
and since we have $H_\eps\big(x,(\frac{1}{2}(a(x)-1)^+,\infty)\big)=\big(\!-\frac{1}{4}[(a(x)-1)^+]^2,\infty\big)$ for all $x\in\R^N$, it follows that
the map $H_\eps(x,\cdot)$ defines a homeomorphism from $\big[\frac{1}{2}(a(x)-1)^+,\infty\big)$ to $\big[\!-\frac{1}{4}[(a(x)-1)^+]^2,\infty\big)$ for all $x\in\R^N$ and $\eps>0$.
Let us denote by $\Theta_{x,\eps}:\big[\!-\frac{1}{4}[(a(x)-1)^+]^2,\infty\big)\to\big[\frac{1}{2}(a(x)-1)^+,\infty\big)$ its reciprocal, that is, $\Theta_{x,\eps}(H_\eps(x,t))=t$ for all $\eps>0$, $x\in\R^N$ and $t\in\big[\frac{1}{2}(a(x)-1)^+,\infty\big)$.

Fix $y_0\in\R^N\setminus\mathscr{N}$. For every $x\in\R^N\setminus\mathscr{N}$, we have, by Lemma~\ref{LE:LOWER}, that $u_\eps(x)\in\big((a(x)-1)^+,\infty\big)\subset\big[\frac{1}{2}(a(x)-1)^+,\infty\big)$, hence~\eqref{gxy2} yields
$$ H_\eps(y_0,u(y_0))+\mathcal{J}_\eps(x,y_0)=H_\eps(x,u_\eps(x))\in \left[-\frac{1}{4}[(a(x)-1)^+]^2, \infty\right).$$
Since the function $x\mapsto H_\eps(y_0,u_\eps(y_0))+\mathcal{J}_\eps(x,y_0)$ is continuous (in the whole space $\R^N$), since $H_\eps$ is itself continuous in $\R^N\times\R$ and since $\mathscr{N}$ is negligible, it follows that $H_\eps(y_0,u_\eps(y_0))+\mathcal{J}_\eps(x,y_0)\in\big[\!-\frac{1}{4}[(a(x)-1)^+]^2,\infty\big)$ for all $x\in\R^N$. Thus, we are allowed to define
$$ \widetilde{u}_\eps(x)=\Theta_{x,\eps}\big(H_\eps(y_0,u_\eps(y_0))+\mathcal{J}_\eps(x,y_0)\big) \text{ for }x\in\R^N.$$
By~\eqref{gxy2}, one has $\widetilde{u}_\eps=u_\eps$ in $\R^N\setminus\mathscr{N}$. Furthermore, $\widetilde{u}_\eps$ is continuous in $\R^N$ owing to its definition, since $\mathcal{J}_\eps$ is continuous in $\R^N\times\R^N$ and $(x,s)\mapsto \Theta_{x,\eps}(s)$ is continuous in the set $\big\{(x,s)\in \R^N\times\R;\,s\in[-\frac{1}{4}[(a(x)-1)^+]^2,\infty)\big\}$. Even if it means redefining $u_\eps$ by $\widetilde{u}_\eps$ in $\R^N$, it follows that $u_\eps$ is continuous in $\R^N$ and that~\eqref{gxy2} (resp. \eqref{KPP}) holds, by continuity, for all $x,y\in\R^N$ (resp. for all $x\in\R^N$).

The case $0<m\leq2$ is similar but technically simpler since $\partial_sH_\eps(x,s)=1-\eps^m[a(x)-2s]>0$ for all $0<\eps<\|a^+\|_{\infty}^{-1/m}$ and all $s\in[0,\infty)$. This means that, for all $x\in\R^N$ and all $0<\eps<\|a^+\|_{\infty}^{-1/m}$, the map $H_\eps(x,\cdot)$ defines a diffeomorphism from $[0,\infty)$ to $[0,\infty)$. From here, with the same arguments as above, we deduce that $u_\eps$ may be redefined up to a negligible set as a continuous function in $\R^N$.

Lastly, let us show that $\limsup_{|x|\to\infty}u_\eps(x)=0$. To this end, we notice that since $\limsup_{|x|\to\infty}a(x)<0$ (by assumption), there exists then $R>0$ such that $a(x)\leq0$ for any $x\in\R^N\setminus B_R$. Whence, it follows from the equation satisfied by $u_\eps$ that
%
\begin{align}
J_\eps\ast u_\eps(x)\geq\eps^m\hspace{0.1em}u_\eps(x)^2+u_\eps(x)\geq u_\eps(x) \text{ for all }x\in\R^N\setminus B_R. \label{bla}
\end{align}
But since $u_\eps\in L^\infty(\R^N)$, we may apply the reverse Fatou inequality so to obtain
$$ 0\leq M:=\limsup_{|x|\to\infty}\hspace{0.1em} u_\eps(x)\leq \limsup_{|x|\to\infty}\int_{\R^N}J_\eps(y)\hspace{0.1em}u_\eps(x-y)\hspace{0.1em}\mathrm{d}y\leq M. $$
Passing to the limit superior in \eqref{bla}, we find that $M\geq \eps^m\hspace{0.1em}M^2+M$. Therefore $M=0$, which thereby completes the proof of Proposition~\ref{LE:REG:POS}.
\end{proof}


\section{Construction of sub- and super-solution}\label{SE:SUB}

This section is devoted to the construction of global sub- and super-solution to \eqref{KPP}.

\begin{lemma}\label{LE:SUBSOL}
Let $0\leq m<2$. Assume \eqref{HYP:A1} and \eqref{HYP:Jweak}. Then, for all $\theta\in(0,1)$ and all $z\in\mathrm{supp}(a^+)$, there is a neighborhood $V_{z,\theta}\subset\mathrm{supp}(a^+)$ of $z$, a number $\eps_{z,\theta}>0$ and a nonnegative function $\underline{u}^{z,\theta}\in C_c(\R^N)$ satisfying
$$ \mathrm{supp}(\underline{u}^{z,\theta})= V_{z,\theta},\ \underline{u}^{z,\theta}(z)\geq(1-\theta)\hspace{0.1em}a^+(z) \text{ and }\underline{u}^{z,\theta}(x)<a^+(x)\text{ for all }x\in V_{z,\theta}, $$
and such that $\underline{u}^{z,\theta}$ is a sub-solution to \eqref{KPP} for all $0<\eps<\eps_{z,\theta}$.
\end{lemma}
\begin{proof}
Let us fix $\theta\in(0,1)$ and an arbitrary point $z$ in $\mathrm{supp}(a^+)$. Up to immaterial translations, we may assume (without loss of generality) that $z=0$. 

\smallskip

Since $a\in C(\R^N)$, we may find some $R>0$ such that $(1-\theta/4)\hspace{0.1em}a^+(0)\hspace{0.1em}\mathds{1}_{B_{2R}}\leq a^+$.
Now, let $\eta\in C_c^2(\R^N)$ be such that $(1-\theta)\hspace{0.1em}a^+(0)\hspace{0.1em}\mathds{1}_{B_{R/2}}\leq \eta\leq (1-\theta/2)\hspace{0.1em}a^+(0)\hspace{0.1em}\mathds{1}_{B_{R}}$ and that $\mathrm{supp}(\eta)=B_R$. 
Furthermore, let $\psi\in C^2([0,\infty))$ be given by
$$ \psi(r):=\max\left\{(1-r)^3,0\right\},\text{ for }r\geq0, $$
and let $\Psi\in C^2(\R^N\setminus\{0\})$ be the function given by $\Psi(x):=\psi(|x|)$. Since $\psi'(r)=-3(1-r)^2\mathds{1}_{[0,1]}(r)$ and $\psi''(r)=6(1-r)\mathds{1}_{[0,1]}(r)$, by computing the Laplacian of $\Psi$ we obtain that
$$ \Delta\Psi(x)=\psi''(|x|)+\frac{N-1}{|x|}\hspace{0.1em}\psi'(|x|)=3(1-|x|)\left(2-\frac{N-1}{|x|}(1-|x|)\right)\mathds{1}_{(0,1]}(|x|)\geq0, $$
for all $x\neq0$ with $|x|\geq\frac{N-1}{N+1}$.
In particular, $\Psi$ is subharmonic in $\R^N\setminus \overline{B_{(N-1)/(N+1)}}$. Let
$$ \varphi(x):=C_\varkappa\hspace{0.1em}\min\left\{\Psi(x),\psi(\varkappa)\right\}\text{ and }\varkappa:=\max\left\{\frac{1}{2},\frac{N-1}{N+1}\right\}\in(0,1), $$
where $C_\varkappa>0$ is a constant such that $\varphi$ has unit mass.
Clearly, $\varphi$ is a radial nonnegative continuous piecewise $C^2$ function with unit mass, supported in $B_1$ and subharmonic in $\R^N\setminus \overline{B_{\varkappa}}$.
Finally, for all $x\in\R^N$, we set
$$ \underline{u}(x):=\eta\ast\varphi_{R/2}(x), $$
where $\varphi_{R/2}(x)=(R/2)^{-N}\varphi\left(2x/R\right)$ for all $x\in\R^N$.
Observe immediately that $\underline{u}\in C_c^2(\R^N)$ is supported in $B_{3R/2}$.
Moreover, since $x-y\in\R^N\setminus \overline{B_{R\varkappa/2}}$ for any $(x,y)\in\R^N\setminus \overline{B_{R(1+\varkappa/2)}}\times B_R$, since $\varphi_{R/2}$ is $C^2$ and subharmonic in $\R^N\setminus \overline{B_{R\varkappa/2}}$ and since $\eta\geq0$, we have
$$ \Delta\hspace{0.05em}\underline{u}(x)=\int_{B_R}\eta(y)\hspace{0.1em}\Delta\varphi_{R/2}(x-y)\hspace{0.1em}\mathrm{d}y\geq0\text{ for all }x\in\R^N\setminus\overline{B_{R(1+\varkappa/2)}}. $$
Hence, $\underline{u}$ is subharmonic in $\R^N\setminus\overline{B_{R(1+\varkappa/2)}}$. Moreover, by construction of $\eta$, we have
$$ \underline{u}(0)=\int_{\R^N}\varphi_{R/2}(y)\hspace{0.05em}\eta(-y)\hspace{0.1em}\mathrm{d}y\geq \hspace{0.1em}(1-\theta)\hspace{0.1em}a^+(0)\int_{B_{R/2}}\varphi_{R/2}(y)\hspace{0.1em}\mathrm{d}y=(1-\theta)\hspace{0.1em}a^+(0). $$
Similarly, we have
$$ \underline{u}(x)\leq \left(1-\frac{\theta}{2}\right)a^+(0)\hspace{0.1em}\mathds{1}_{B_{3R/2}}(x)<a^+(x) \text{ for all }x\in B_{3R/2}. $$

\smallskip

To complete the proof we only need to show that the function $\underline{u}$ is a sub-solution to \eqref{KPP} provided $\eps>0$ is small enough. For it, we first observe that, since $\eta\leq (1-\theta/2)\hspace{0.1em}a^+(0)\hspace{0.1em}\mathds{1}_{B_R}$ in $\R^N$ (by construction of $\eta$), we have
\begin{align*}
\underline{u}(x)\left(a(x)-\underline{u}(x)\right)\geq\underline{u}(x)\left(a(x)-\left(1-\frac{\theta}{2}\right)a^+(0)\int_{\R^N}\varphi(y)\hspace{0.1em}\mathds{1}_{B_R}(x-R y/2)\hspace{0.1em}\mathrm{d}y\right).
\end{align*}
Since $\varphi$ has unit mass, since $\underline{u}$ is supported in $B_{3R/2}$ and since $(1-\theta/4)\hspace{0.1em}a^+(0)\hspace{0.1em}\mathds{1}_{B_{2R}}\leq a^+$ (by construction of $R$), we get
\begin{align}
\underline{u}(x)\left(a(x)-\underline{u}(x)\right)&\geq \underline{u}(x)\left(a^+(x)-\left(1-\frac{\theta}{2}\right)a^+(0)\right) \nonumber \\
&\geq\underline{u}(x)\left(\left(1-\frac{\theta}{4}\right)a^+(0)-\left(1-\frac{\theta}{2}\right)a^+(0)\right)=\frac{\theta}{4}\,a^+(0)\,\underline{u}(x). \label{KKl0}
\end{align}
Furthermore, letting $R_\eps:=(1-\varkappa)R/(8\eps)$, we have $0<R_\eps\hspace{0.05em}\eps <(1-\varkappa)R/4$ for any $\eps>0$ (by construction of $R_\eps$). It follows that $B_{R_\eps \eps}(x)\subset \R^N\setminus\overline{B_{R(1+\varkappa/2)}}$ for any $x\in \R^N\setminus B_{R(1+(\varkappa+1)/4)}\,(\subset \R^N\setminus\overline{B_{R(1+\varkappa/2)}})$. But, since $\Delta\hspace{0.05em}\underline{u}\geq0$ in $\R^N\setminus\overline{B_{R(1+\varkappa/2)}}$, we may apply the mean value inequality for subharmonic functions in the ball $B_{R_\eps \eps}(x)$ (see e.g. \cite[Theorem 2.1, p.14]{Gilbarg}) and we get
$$ \dashint_{\mathbb{S}^{N-1}}\underline{u}(x+r\hspace{0.05em}e)\hspace{0.1em}\mathrm{d}\mathcal{H}^{N-1}(e)=\dashint_{\partial B_r(x)}\underline{u}(e)\hspace{0.1em}\mathrm{d}\mathcal{H}^{N-1}(e)\geq \underline{u}(x), $$
for all $x\in \R^N\setminus B_{R(1+(\varkappa+1)/4)}$ and all $0<r\leq R_\eps\hspace{0.05em}\eps$.
As a consequence, there holds
\begin{align}
\int_{B_{R_\eps \eps}(x)}J_\eps(x-y)\hspace{0.1em}\underline{u}(y)\hspace{0.1em}\mathrm{d}y&=\int_0^{R_\eps}J_0(t)\hspace{0.1em}t^{N-1}\left(\int_{\mathbb{S}^{N-1}}\underline{u}(x+\eps t\hspace{0.1em}e)\hspace{0.1em}\mathrm{d}\mathcal{H}^{N-1}(e)\right)\mathrm{d}t \nonumber \\
&\geq \sigma_N\hspace{0.1em}\underline{u}(x)\int_0^{R_\eps}J_0(t)\hspace{0.1em}t^{N-1}\mathrm{d}t=\underline{u}(x)\int_{B_{R_\eps \eps}(x)}J_\eps(x-y)\hspace{0.1em}\mathrm{d}y, \label{KKl01}
\end{align}
for all $x\in \R^N\setminus B_{R(1+(\varkappa+1)/4)}$ and all $\eps>0$, where $J_0\in L_{\mathrm{loc}}^1(0,\infty)$ is a function such that $J(x)=J_0(|x|)$ for a.e. $x\in\R^N$.
On the other hand, for any $x\in\R^N$, we have
\begin{align*}
\int_{\R^N\setminus B_{R_\eps\eps}(x)}J_\eps(x-y)\hspace{0.1em}\mathrm{d}y&=\int_{\R^N\setminus B_{R_\eps}}J(y)\hspace{0.1em}\mathrm{d}y\leq \frac{8^m\eps^m}{(1-\varkappa)^mR^m}\int_{\R^N\setminus B_{R_\eps}}J(y)|y|^m\hspace{0.05em}\mathrm{d}y,
\end{align*}
as follows from a direct application of Markov's inequality.
Since the $m$-th order moment of $J$ is finite and since $R_\eps\to\infty$ as $\eps\to0^+$, there is then some $\eps_1>0$ such that
\begin{align*}
\frac{8^m}{(1-\varkappa)^mR^m}\int_{\R^N\setminus B_{R_\eps}}J(y)|y|^m\hspace{0.05em}\mathrm{d}y\leq\frac{\theta}{8}\,a^+(0),
\end{align*}
for all $0<\eps<\eps_1$.
Consequently, for any $x\in\R^N$ and any $0<\eps<\eps_1$, we get
\begin{align}
\int_{\R^N\setminus B_{R_\eps \eps}(x)}J_\eps(x-y)\hspace{0.1em}\underline{u}(y)\hspace{0.1em}\mathrm{d}y-\underline{u}(x)\int_{\R^N\setminus B_{R_\eps \eps}(x)}J_\eps(x-y)\hspace{0.1em}\mathrm{d}y\geq-\eps^m\hspace{0.1em}\frac{\theta}{8}\,a^+(0)\,\underline{u}(x).  \label{KKKK2}
\end{align}
Hence, collecting \eqref{KKl0}, \eqref{KKl01} and \eqref{KKKK2}, we obtain
\begin{align}
\frac{1}{\eps^m}(J_\eps\ast\underline{u}(x)-\underline{u}(x))+\underline{u}(x)(a(x)-\underline{u}(x))\geq \frac{\theta}{8}\,a^+(0)\,\underline{u}(x)\geq0, \label{KKl3}
\end{align}
for any $x\in\R^N\setminus B_{R(1+(\varkappa+1)/4)}$ and any $0<\eps<\eps_{1}$.

\smallskip

Let us now estimate $\eps^{-m}(J_\eps\ast\underline{u}-\underline{u})$ in $B_{R(1+(\varkappa+1)/4)}$.
To this end, we observe that
\begin{align} \frac{1}{\eps^m}(J_\eps\ast\underline{u}(x)-\underline{u}(x))=\frac{M_{m}(J)}{2}\int_{\R^N}\rho_\eps(y)\hspace{0.1em}\frac{\Delta_y^{2}\hspace{0.1em}\underline{u}(x)}{|y|^{m}}\hspace{0.1em}\mathrm{d}y, \label{KKl1m}
\end{align}
where $\Delta_y^{2}\hspace{0.1em}\underline{u}(x):=\underline{u}(x+y)-2\hspace{0.1em}\underline{u}(x)+\underline{u}(x-y)$ and $\rho_\eps$ is given by
\begin{align}
\rho_\eps(y):=\frac{1}{\eps^N}\hspace{0.1em}\rho\left(\frac{y}{\eps}\right) \text{ where } \rho(y):=\frac{J(y)|y|^{m}}{M_{m}(J)}. \label{mollifier}
\end{align}
Using \cite[Proposition 6.1]{JB} and recalling that $\underline{u}\in C_c^2(\R^N)$ and $0\leq m<2$, we have
\begin{align*} \limsup_{\eps\to0^+}\int_{\R^N}\rho_\eps(y)\hspace{0.1em}\frac{\|\Delta_y^{2}\hspace{0.1em}\underline{u}\|_{L^\infty(\R^N)}}{|y|^{m}}\hspace{0.1em}\mathrm{d}y \leq \limsup_{|h|\to0}\frac{\|\Delta_h^2\hspace{0.1em}\underline{u}\|_{L^\infty(\R^N)}}{|h|^m}\leq\lim_{|h|\to0}\|D^2\underline{u}\|_{L^\infty(\R^N)}|h|^{2-m}=0. 
\end{align*}
But since $\inf_{B_{R(1+(\varkappa+1)/4)}}\underline{u}>0$, there is then some $\eps_2>0$ (independent of $x$) such that
$$ \frac{M_{m}(J)}{2}\int_{\R^N}\rho_\eps(y)\hspace{0.1em}\frac{|\Delta_y^{2}\hspace{0.1em}\underline{u}(x)|}{|y|^{m}}\hspace{0.1em}\mathrm{d}y\leq\frac{\theta}{8}\,a^+(0)\inf_{B_{R(1+(\varkappa+1)/4)}}\hspace{0.1em}\underline{u}, $$
for all $x\in B_{R(1+(\varkappa+1)/4)}$ and all $0<\eps<\eps_2$.
Therefore, recalling \eqref{KKl0} and \eqref{KKl1m}, we get
\begin{align}
\frac{1}{\eps^m}(J_\eps\ast\underline{u}-\underline{u})+\underline{u}(a-\underline{u})\geq \frac{\theta}{8}\,a^+(0)\inf_{B_{R(1+(\varkappa+1)/4)}}\hspace{0.1em}\underline{u}\geq0, \label{KKl4}
\end{align}
in $B_{R(1+(\varkappa+1)/4)}$ for all $0<\eps<\eps_{2}$.
By \eqref{KKl3} and \eqref{KKl4}, we obtain that $\underline{u}$ is a continuous, nonnegative sub-solution to \eqref{KPP} for all $0<\eps<\eps_{0}=\min\{\eps_1,\eps_2\}$.
\end{proof}

Let us now construct a global super-solution to \eqref{KPP}.

\begin{lemma}\label{LE:SUPER}
Let $0\leq m\leq 2$ and let $\beta>0$. Assume \eqref{HYP:A1} and that $J\in L^1(\R^N)$ is a nonnegative radial kernel with unit mass. Suppose that $J$ has a finite $\beta$-th order moment. Then, there exist a constant $C_{\eps,\beta}>0$ and a positive function $\overline{u}_{\eps,\beta}\in C_0(\R^N)$ such that
\begin{align}
a^+(x)\leq\overline{u}_{\eps,\beta}(x)\leq \frac{C_{\eps,\beta}}{1+|x|^\beta}\text{ for all }x\in\R^N, \label{estimate:supersol}
\end{align}
and $\overline{u}_{\eps,\beta}$ is a super-solution to \eqref{KPP} for all $\eps>0$.
\end{lemma}
\begin{proof}
The proof of Lemma~\ref{LE:SUPER} follows the same line of ideas as in \cite{BCV}. We begin by introducing some notations. First, we denote by $\mathcal{M}_{\eps,m}$ the operator given by
$$ \mathcal{M}_{\eps,m}[\varphi](x):=\frac{1}{\eps^m}\int_{\R^N}J_\eps(x-y)(\varphi(y)-\varphi(x))\hspace{0.1em}\mathrm{d}y. $$
Second, we let $R(a^+)\geq 1$ be such that $\mathrm{supp}(a^+)\subset B_{R(a^+)}$. Next, we let $R\geq R(a^+)$ and $\ell>0$ be such that $a(x)\leq -\ell \text{ for all }|x|\geq R$ (these numbers are guaranteed to exist by assumption \eqref{HYP:A1}).
Now, given $0<\tau<1$, we let $C_{\tau,R}:=(1/\tau+R^\beta)\|a^+\|_\infty$ and 
$$ \overline{u}(x):=\left\{\begin{array}{cl} \displaystyle\frac{C_{\tau,R}\hspace{0.1em}\tau}{1+\tau\hspace{0.025em}R^\beta} & \text{if }x\in B_{R}, \vspace{3pt}\\ \displaystyle\frac{C_{\tau,R}\hspace{0.1em}\tau}{1+\tau\hspace{0.025em}|x|^\beta} & \text{if }x\in \R^N\setminus B_{R}. \end{array}\right. $$

Our goal will be to prove that there exist $R_{\eps,\beta}>0$ and $\tau_{\eps,\beta}>0$ such that $\overline{u}$ is a super-solution to \eqref{KPP} for all $R\geq \max\{R_{\eps,\beta},R(a^+)\}$, all $0<\tau<\tau_{\eps,\beta}$ and all $\eps>0$.
Readily, we observe that, by construction of $C_{\tau,R}$ and since $R\geq R(a^+)$, we have
\begin{align}
\mathcal{M}_{\eps,m}[\overline{u}]+\overline{u}(a-\overline{u})\leq 0\text{ in }B_{R}. \label{interieur}
\end{align}
To complete the proof it suffices to prove that this still holds on $\R^N\setminus B_{R}$. To this end, we introduce the auxiliary function $U(x):=C_{\tau,R}\hspace{0.1em}\tau(1+\tau|x|^\beta)^{-1}$, $x\in\R^N$, and we remark that
\begin{align*}
\mathcal{M}_{\eps,m}[\overline{u}](x)+\overline{u}(x)(a(x)-\overline{u}(x))\leq \mathcal{M}_{\eps,m}[U](x)-\ell\,U(x),
\end{align*}
for all $|x|\geq R$ (by construction of $\ell$ and $R$). Developing this results in
\begin{align}
\mathcal{M}_{\eps,m}[\overline{u}](x)+\overline{u}(x)(a(x)-\overline{u}(x))&\leq \overline{u}(x)\left(\frac{1}{\eps^m}\int_{\R^N}J_\eps(y)\left(\frac{U(x+y)}{U(x)}-1\right)\mathrm{d}y-\ell\right) \nonumber \\
&=\overline{u}(x)\left(\frac{\tau}{\eps^m}\int_{\R^N}\frac{J_\eps(y)(|x|^\beta-|x+y|^\beta)}{1+\tau|x+y|^\beta}\hspace{0.1em}\mathrm{d}y-\ell\right). \label{II0}
\end{align}
Let us split the integral on the right-hand side as
$$ \int_{\R^N}\frac{J_\eps(y)(|x|^\beta-|x+y|^\beta)}{1+\tau|x+y|^\beta}\hspace{0.1em}\mathrm{d}y=\int_{|y|\geq|x|/2}\cdots\,+\int_{|y|<|x|/2}\cdots\,=:I_1+I_2. $$
Clearly,
\begin{align}
I_1\leq 2^\beta\int_{\R^N}J_\eps(y)|y|^\beta\mathrm{d}y=(2\eps)^\beta M_\beta(J). \label{II1}
\end{align}
Let us now estimate $I_2$. We will estimate the integrand by a quantity that does not depend on $x$. So let $x,y\in\R^N$ be such that $|x|\geq R$ and $|y|<|x|/2$. Let $p:=\lfloor \beta\rfloor+1$ and $q:=p/\beta$, where $\lfloor\cdot\rfloor$ denotes the floor function. Using the binomial formula, we have
\begin{align*}
\tau^q\hspace{0.1em}\frac{|x|^p-|x+y|^p}{1+\tau^q|x+y|^p}&\leq \tau^q\hspace{0.1em}\frac{|x|^p-(|x|-|y|)^p}{1+\tau^q|x+y|^p}= \tau^q\,\sum_{k=1}^p\binom{p}{k}\frac{(-1)^{k+1}|x|^{p-k}|y|^k}{1+\tau^q|x+y|^p}.
\end{align*}
Since $|x|\geq R\geq1$ and $(1+\tau^q\hspace{0.1em}|x+y|^p)^{-1}\leq 2^p(2^p+\tau^q\hspace{0.1em}|x|^p)^{-1}$, we further get
\begin{align*}
\tau^q\hspace{0.1em}\frac{|x|^p-|x+y|^p}{1+\tau^q|x+y|^p}&\leq \frac{2^p}{R}\sum_{k=1}^p\binom{p}{k}\hspace{0.1em}|y|^k\hspace{0.1em}\frac{\tau^q|x|^{p}}{2^p+\tau^q|x|^p} \leq \frac{2^p}{R}\sum_{k=0}^p\binom{p}{k}\hspace{0.1em}|y|^k.
\end{align*}
Using the binomial formula once again, we arrive at
$$ \frac{2^p}{R}(1+|y|)^p\geq \tau^q\hspace{0.1em}\frac{|x|^p-|x+y|^p}{1+\tau^q|x+y|^p}\text{ for all }|x|\geq R\text{ and all }|y|<\frac{|x|}{2}. $$
Moreover, since the function $t\in[0,\infty)\mapsto t^{1/q}$ is concave it is in particular subadditive. Consequently, we have that $(1+\tau^q|x+y|^p)^{1/q}\leq 1+\tau|x+y|^{p/q}$ and that $(|x|^p-|x+y|^p)^{1/q}\geq |x|^{p/q}-|x+y|^{p/q}$ whenever $y\in \overline{B_{|x|}(-x)}$. Using this we deduce that
\begin{align}
{2^\beta}R^{-\frac{1}{q}}(1+|y|)^\beta={2^\frac{p}{q}}{R^{-\frac{1}{q}}}(1+|y|)^\frac{p}{q}\geq \tau\hspace{0.1em}\frac{|x|^\frac{p}{q}-|x+y|^\frac{p}{q}}{1+\tau|x+y|^\frac{p}{q}}=\tau\hspace{0.1em}\frac{|x|^\beta-|x+y|^\beta}{1+\tau|x+y|^\beta}, \label{ineq}
\end{align}
for all $|x|\geq R$ and all $y\in \overline{B_{|x|}(-x)}$ with $|y|<|x|/2$. Notice that this remains true if $y\in\R^N\setminus \overline{B_{|x|}(-x)}$ since the right-hand side in \eqref{ineq} is negative. Hence, \eqref{ineq} holds for all $|x|\geq R$ and all $y\in \R^N$ with $|y|<|x|/2$. As a consequence, we get
\begin{align}
I_2\leq \frac{2^\beta R^{-\frac{1}{q}}}{\tau}\int_{\R^N}J_\eps(y)(1+|y|)^\beta\hspace{0.1em}\mathrm{d}y. \label{II2}
\end{align}
Plugging \eqref{II1} and \eqref{II2} in \eqref{II0}, we find that
$$ \mathcal{M}_{\eps,m}[\overline{u}](x)+\overline{u}(x)(a(x)-\overline{u}(x))\leq \overline{u}(x)\left(C_1\hspace{0.1em}\tau+C_2\hspace{0.1em}R^{-\frac{1}{q}}-\ell\right), $$
for all $|x|\geq R$ and some $C_1,C_2>0$ depending on $\eps$, $m$, $J$ and $\beta$.
Hence, there exist $0<\tau_{\eps,\beta}<1$ and $R_{\eps,\beta}>0$ such that, for all $R\geq R_{\eps,\beta}$ and all $0<\tau<\tau_{\eps,\beta}$, we have
$$ \mathcal{M}_{\eps,m}[\overline{u}](x)+\overline{u}(x)(a(x)-\overline{u}(x))\leq -\frac{1}{2}\hspace{0.1em}\overline{u}(x)\hspace{0.1em}\ell<0 \text{ for all }|x|\geq R. $$
Recalling \eqref{interieur}, we obtain that $\overline{u}$ is indeed a super-solution to \eqref{KPP}. Moreover, estimate \eqref{estimate:supersol} is trivially satisfied. This completes the proof.
\end{proof}

\section{Existence, uniqueness and asymptotic analysis}\label{SE:PROOF}

In this section, we prove our main results Theorems~\ref{THEO}-\ref{THEO02}. Our strategy follows some ideas already used in \cite{BCV,BCHV,BC,Cov07} and relies on the well-known monotone iterative method together with Lemmata~\ref{LE:SUBSOL}-\ref{LE:SUPER}. 
As in \cite{BCV}, we will construct the unique positive minimal solution to \eqref{KPP} as the pointwise limit as $R\to\infty$ of the unique positive solution to
\begin{align}
\mathcal{M}_{R,\eps,m}[u](x)+u(x)(a(x)-u(x))=0\text{ for }x\in B_R, \label{truncated}
\end{align}
where $\mathcal{M}_{R,\eps,m}$ is the operator given by
\begin{align*}
\mathcal{M}_{R,\eps,m}[\varphi](x):=\frac{1}{\eps^m}\left(\int_{B_R}J_\eps(x-y)\hspace{0.1em}\varphi(y)\hspace{0.1em}\mathrm{d}y-\varphi(x)\right). 
\end{align*}
To this end, we need to construct such a solution and to establish its uniqueness. In view of this, we first prove a few comparison principles for the problem \eqref{truncated}. This will be done in the next subsection. Once this is done, we will be in position to prove Theorems~\ref{THEO}-\ref{THEO02} using the sub- and the super-solution constructed at Section~\ref{SE:SUB}.

\subsection{Comparison principles}\label{SE:COMP}

Let us list in this subsection, some comparison principles which will allow us to establish Theorem~\ref{THEO}.

\begin{lemma}[Maximum principle]\label{LE:comparaison}
Let $J\in L^1(\R^N)$ be a nonnegative function. Let $k>0$ and $w\in C(\overline{B_R})$ be such that
\begin{align}
\mathcal{M}_{R,\eps,m}[w]-k\hspace{0.1em}w\geq0\text{ in }B_R. \label{Lemme:prem}
\end{align}
Then, $w\leq 0 \text{ in }B_R$.
\end{lemma}
\begin{proof}
Suppose, by contradiction, that $\sup_{B_R}w>0$. Since $w$ is continuous, there exists a point $\bar{x}\in\overline{B_R}$ such that $w(\bar{x})=\sup_{B_R}\hspace{0.1em}w>0$.
Let $(x_j)_{j\geq0}\subset B_R$ be such that $x_j\to\bar{x}$ as $j\to\infty$.
Since $k>0$, specializing \eqref{Lemme:prem} at $x_j$ and passing to the limit as $j\to\infty$ (using the dominated convergence theorem), we obtain
\begin{align*}
0&\geq \frac{1}{\eps^m}\int_{B_R}J_\eps(\bar{x}-y)\left(w(y)-\sup_{B_R}\hspace{0.1em}w\right)\mathrm{d}y\geq \left(k+\frac{1}{\eps^m}\int_{\R^N\setminus B_R}J_\eps(\bar{x}-y)\hspace{0.1em}\mathrm{d}y\right)\sup_{B_R}\hspace{0.1em}w>0,
\end{align*}
which is a contradiction. The proof is thereby complete.
\end{proof}

\begin{lemma}[Comparison principle]\label{LE:COMPARISON:BR}
Let $J\in L^1(\R^N)$ be a nonnegative function. Let $v\in C(\overline{B_R})$ be a positive function and let $u\in C(\overline{B_R})$ be a nonnegative function. Suppose that $u$ and $v$ are a sub- and a super-solution to \eqref{truncated}, respectively.
Then, $u\leq v$ in $\overline{B_R}$.
\end{lemma}
\begin{proof}
The proof follows from some ideas due to Coville (see \cite[Section 6.3]{Cov10}). Let $u\in C(\overline{B_R})$ be a nonnegative sub-solution and let $v\in C(\overline{B_R})$ be a positive super-solution. Define
$$ \gamma^*:=\inf\left\{\gamma>0\text{ s.t. }\gamma\hspace{0.1em}v\geq u\text{ in }\overline{B_R}\right\}. $$
Suppose, by contradiction, that $\gamma^*>1$. Then, we have
\begin{align}
\mathcal{M}_{R,\eps,m}[\gamma^*v]+\gamma^*v(a-\gamma^*v)\leq\gamma^*v(a-\gamma^*v)-\gamma^*v(a-v)
&=\gamma^*(1-\gamma^*)v^2<0. \label{kontrdg}
\end{align}
Since $\overline{B_R}$ is compact, by minimality of $\gamma^*$ there must be some $x_0\in\overline{B_R}$ such that $\gamma^*v(x_0)=u(x_0)$. Hence, there holds
$$ \mathcal{M}_{R,\eps,m}[\gamma^*v](x_0)+\gamma^*v(x_0)(a(x_0)-\gamma^*v(x_0))\geq \frac{1}{\eps^m}\int_{B_R}J_\eps(x_0-y)\hspace{0.1em}(\gamma^*v(y)-u(y))\hspace{0.1em}\mathrm{d}y. $$
By assumption, the integrand on the right-hand side is nonnegative. This contradicts the strict inequality in \eqref{kontrdg}. Therefore, $\gamma^*\leq1$ which ensures that $u\leq v$ in $\overline{B_R}$.
\end{proof}

\subsection{Proof of Theorem~\ref{THEO}}

We now proceed to the proof of Theorem~\ref{THEO}. For the convenience of the reader, the proof is split into four parts. After a preparatory step, we establish the existence of a unique positive solution to the truncated problem \eqref{truncated} using the sub-solution constructed at Section~\ref{SE:SUB} and the comparison principles of Subsection~\ref{SE:COMP}. Then, we use this solution to prove that \eqref{KPP} admits a unique positive minimal solution. Finally, exploiting the properties of the sub-solution constructed at Section~\ref{SE:SUB} we derive its asymptotic behavior, which will thereby complete the proof of Theorem~\ref{THEO}.
\vskip 0.3cm

\noindent\emph{Step 1. Preliminaries}
\smallskip

Let $R_0>0$ be such that $\overline{\mathrm{supp}(a^+)}\subset B_{R_0}$ and let $R\geq R_0$ be fixed. Observe immediately that if $\psi\in C_c(\R^N)$ is a nonnegative sub-solution to \eqref{KPP} satisfying $\mathrm{supp}(\psi)\subset\mathrm{supp}(a^+)$, then $\psi$ is also a nonnegative sub-solution to the truncated problem \eqref{truncated} for all $R\geq R_0$.

Now, let $z\in\mathrm{supp}(a^+)$ and $\theta\in(0,1)$ be arbitrary. By Lemma~\ref{LE:SUBSOL}, there exist $\eps_{z,\theta}>0$, a neighborhood $V_{z,\theta}\subset\mathrm{supp}(a^+)$ of $z$ and a sub-solution $\underline{u}^{z,\theta}\in C_c(\R^N)$ to \eqref{KPP} satisfying
$$ \mathrm{supp}(\underline{u}^{z,\theta})=V_{z,\theta},\ \underline{u}^{z,\theta}(z)\geq(1-\theta)\hspace{0.1em}a^+(z) \text{ and } \underline{u}^{z,\theta}(x)<a^+(x)\text{ for all }x\in V_{z,\theta}. $$
Let $\overline{u}:=\|a^+\|_\infty$. Then, $\underline{u}^{z,\theta}$ and $\overline{u}$ are continuous global sub- and super-solution to \eqref{KPP} for all $0<\eps<\eps_{z,\theta}$. Since $R\geq R_0$, these functions are also a sub- and a super-solution to the truncated problem \eqref{truncated} for all $0<\eps<\eps_{z,\theta}$.
\vskip 0.3cm

\noindent\emph{Step 2. Construction of a unique solution to the truncated problem}
\smallskip

Let $0<\eps<\min\{\eps_{z,\theta},\eps_{m,a^+}\}$, where $\eps_{m,a^+}>0$ has the same meaning as in Proposition~\ref{LE:REG:POS}.
Let us first observe that, for all $R>0$, the operator $\mathcal{M}_{R,\eps,m}$ is linear and continuous on $\big(C(\overline{B_R}),\left\|\cdot\right\|_\infty\big)$.
Let $f(x,s):=s(a(x)-s)$ and let $k>0$ be a number large enough so that the map $s\mapsto -k\hspace{0.1em}s-f(x,s)$ is decreasing in $\big[0,2\|a^+\|_\infty\big]$ for all $x\in\R^N$ (this is possible since $a\in L^\infty(\R^N)$) and that $k>\|\mathcal{M}_{R,\eps,m}\|$ so that $k\in\rho(\mathcal{M}_{R,\eps,m})$ where $\rho(\mathcal{M}_{R,\eps,m})$ denotes the resolvent of the operator $\mathcal{M}_{R,\eps,m}$.
\smallskip

We will construct a solution $u_{R,\eps}$ to \eqref{truncated} satisfying $\underline{u}^{z,\theta}\leq u_{R,\eps}\leq\overline{u}$ using a monotone iterative scheme starting from $\underline{u}^{z,\theta}$ and using $\overline{u}$ as a barrier function.
\smallskip

Namely, we set $u_{R,\eps}^0\equiv \underline{u}^{z,\theta}$ and, for $j\geq0$, we let
\begin{align}
\mathcal{M}_{R,\eps,m}[u_{R,\eps}^{j+1}](x)-k\hspace{0.1em}u_{R,\eps}^{j+1}(x)=-k\hspace{0.1em}u_{R,\eps}^{j}(x)-f(x,u_{R,\eps}^j(x)) \text{ for }x\in \overline{B_R}. \label{suite:auxiliaire:schema:iter}
\end{align}
Observe that the $u_{R,\eps}^j$'s are well-defined elements of $C(\overline{B_R})$.
We will show that a solution to \eqref{truncated} can be obtained as the pointwise limit of $(u_{R,\eps}^j)_{j\geq0}$. First, when $j=0$, we have
\begin{align}
\mathcal{M}_{R,\eps,m}[u_{R,\eps}^1](x)-k\hspace{0.1em}u_{R,\eps}^1(x)=-k\hspace{0.1em}u_{R,\eps}^0(x)-f(x,u_{R,\eps}^0(x)) \text{ for }x\in \overline{B_R}. \label{iter:u1u0}
\end{align}
We claim that $\underline{u}=u_{R,\eps}^0\leq u_{R,\eps}^1\leq \overline{u}$ in $\overline{B_R}$. Indeed, we have
\be\left\{\begin{array}{rcl}
\mathcal{M}_{R,\eps,m}[u_{R,\eps}^1-u_{R,\eps}^0](x)-k(u_{R,\eps}^1-u_{R,\eps}^0)\! & \!\!=\!\! &\! -\mathcal{M}_{R,\eps,m}[u_{R,\eps}^0](x)-f(x,u_{R,\eps}^0(x)), \vspace{3pt}\\
\mathcal{M}_{R,\eps,m}[u_{R,\eps}^{1}-\overline{u}](x)-k(u_{R,\eps}^{1}-\overline{u})\! &  \!\!\geq\!\! &\! f(x,\overline{u}(x))+k\hspace{0.1em}\overline{u}(x)-f(x,u_{R,\eps}^0(x))-k\hspace{0.1em}u_{R,\eps}^0(x).\end{array}\right.  \nonumber
\ee
Since $u_{R,\eps}^0=\underline{u}^{z,\theta}$ (resp. $\overline{u}$) is a sub-solution (resp. super-solution) to \eqref{truncated}, $\underline{u}^{z,\theta}\leq \overline{u}$ and $s\mapsto-k\hspace{0.1em}s-f(x,s)$ is decreasing for all $x\in \overline{B_R}$, we obtain that
\begin{align}
\left\{
\begin{array}{rl}
\mathcal{M}_{R,\eps,m}[u_{R,\eps}^1-u_{R,\eps}^0](x)-k(u_{R,\eps}^1-u_{R,\eps}^0)\leq 0,& \vspace{3pt} \\
\mathcal{M}_{R,\eps,m}[u_{R,\eps}^{1}-\overline{u}](x)-k(u_{R,\eps}^{1}-\overline{u})\geq 0.&
\end{array}
\right.\label{sys:iter:u1u0}
\end{align}
This, together with Lemma~\ref{LE:comparaison}, then gives that $\underline{u}^{z,\theta}=u_{R,\eps}^0\leq u_{R,\eps}^1\leq \overline{u}$ in $\overline{B_R}$. Similarly, by \eqref{suite:auxiliaire:schema:iter}, the function $u_{R,\eps}^2\in C(\overline{B_R})$ solves \eqref{iter:u1u0} with $u_{R,\eps}^2$ in place of $u_{R,\eps}^1$ and $u_{R,\eps}^1$ in place of $u_{R,\eps}^0$. Thus, from the monotonicity of $s\mapsto-k\hspace{0.05em}s-f(x,s)$, we deduce that \eqref{sys:iter:u1u0} still holds with $u_{R,\eps}^2$ instead of $u_{R,\eps}^1$ and $u_{R,\eps}^1$ instead of $u_{R,\eps}^0$. We may then apply the maximum principle Lemma~\ref{LE:comparaison} and we deduce that $\underline{u}^{z,\theta}=u_{R,\eps}^0\leq u_{R,\eps}^1\leq u_{R,\eps}^2\leq\overline{u}$ in $\overline{B_R}$. By induction, we infer that the $u_{R,\eps}^j$'s satisfy the monotonicity relation
$$ \underline{u}^{z,\theta}=u_{R,\eps}^0\leq u_{R,\eps}^1\leq u_{R,\eps}^2\leq \cdots\leq u_{R,\eps}^j\leq u_{R,\eps}^{j+1}\leq \cdots \leq\overline{u}. $$
Since $(u_{R,\eps}^j)_{j\geq0}$ is non-decreasing and bounded from above by $\underline{u}$, the function
\begin{align}
u_{R,\eps}(x):=\lim_{j\to\infty}u_{R,\eps}^{j}(x)\in\left[\underline{u}^{z,\theta}(x),\overline{u}(x)\right], \label{definition:de:la:sol}
\end{align}
is well-defined for any $x\in \overline{B_R}$. In particular, it follows from \eqref{definition:de:la:sol} that $u_{R,\eps}$ is nonnegative, not identically zero and bounded. It remains only to check that the function $u_{R,\eps}$ is a solution to \eqref{truncated}. For it, it suffices to let $j\to\infty$ in \eqref{suite:auxiliaire:schema:iter} (using the dominated convergence theorem), which then gives
$$ \mathcal{M}_{R,\eps,m}[u_{R,\eps}](x)+f(x,u_{R,\eps}(x))=0 \text{ for any }x\in \overline{B_R}. $$
Finally, by a straightforward adaptation of Proposition~\ref{LE:REG:POS} we deduce that $u_{R,\eps}\in C(\overline{B_R})$ and that $u_{R,\eps}>0$ in $\overline{B_R}$. Moreover, $u_{R,\eps}$ is the only positive solution to \eqref{truncated} (this follows from Lemma~\ref{LE:COMPARISON:BR}).

\vskip 0.3cm

\noindent\emph{Step 3. Construction of the unique positive minimal solution to \eqref{KPP}}
\smallskip

Now, we observe that, if $R'\geq R$, then the unique positive solution to \eqref{truncated} with $R'$ instead of $R$ will be a super-solution to \eqref{truncated} in the ball $\overline{B_R}$. Hence, using Lemma~\ref{LE:COMPARISON:BR}, we find that $\underline{u}^{z,\theta}\leq u_{R,\eps}\leq u_{R',\eps}\leq \|a^+\|_\infty$ for all $R'\geq R\geq R_0$ and all $0<\eps<\min\{\eps_{z,\theta},\eps_{m,a^+}\}$.
Hence, the function $u_\eps$ given by
$$ u_\eps(x):=\lim_{R\to\infty}\hspace{0.1em}u_{R,\eps}(x)\in\left[\underline{u}^{z,\theta}(x),\|a^+\|_\infty\right], $$
is well-defined and it is nonnegative and bounded. By the dominated convergence theorem, we find that $u_\eps$ is a nontrivial solution to \eqref{KPP} and, by Proposition~\ref{LE:REG:POS}, we further obtain that $u_\eps\in C_0(\R^N)$ and that $u_\eps>0$ in $\R^N$ (remember that $\eps<\eps_{m,a^+}$).

Let us now check that $u_\eps$ is the unique positive minimal solution to \eqref{KPP}. Let $v_\eps\in L^\infty(\R^N)$ be any nontrivial nonnegative solution to \eqref{KPP}. Then, $v_\eps\in C_0(\R^N)$ and $v_\eps>0$ in $\R^N$ (by Proposition~\ref{LE:REG:POS}). But since $v_\eps$ is a super-solution to \eqref{truncated} for all $R\geq R_0$, we have $u_{R,\eps}\leq v_\eps$ in $\overline{B_R}$ for all $R\geq R_0$ (by Lemma~\ref{LE:COMPARISON:BR}), which enforces that $u_\eps\leq v_\eps$. Therefore, $u_\eps$ is indeed the unique positive minimal solution to \eqref{KPP}.
\vskip 0.3cm

\noindent\emph{Step 4. Asymptotics of the minimal solution}
\smallskip

Lastly, by construction of $\underline{u}^{z,\theta}$, we have
\begin{align} \liminf_{\eps\to0^+}\hspace{0.1em}u_\eps(z)\geq \underline{u}^{z,\theta}(z)\geq(1-\theta)\hspace{0.1em}a^+(z). \label{eq:asy:z}
\end{align}
But since $z$ has been chosen arbitrarily, choosing another point $z'\in\mathrm{supp}(a^+)$ would result in the existence of some $\eps_{z',\theta}>0$ and of a solution $\tilde{u}_\eps\in C_0(\R^N)$ to \eqref{KPP} for $0<\eps<\min\{\eps_{z',\theta},\eps_{m,a^+}\}$ which satisfies \eqref{eq:asy:z} with $z'$ instead of $z$ and which is the unique positive minimal solution to \eqref{KPP} in the range $0<\eps<\min\{\eps_{z',\theta},\eps_{m,a^+}\}$. Hence, it coincides with $u_\eps$ for $0<\eps<\min\{\eps_{z,\theta},\eps_{z',\theta},\eps_{m,a^+}\}$. Thus, we have
$$ \liminf_{\eps\to0^+}\hspace{0.1em}u_\eps(z')=\liminf_{\eps\to0^+}\hspace{0.1em}\tilde{u}_\eps(z')\geq \underline{u}^{\theta,z'}(z')\geq(1-\theta)\hspace{0.1em}a^+(z'). $$
Therefore, we find that $\liminf_{\eps\to0^+}\hspace{0.1em}u_\eps(x)\geq (1-\theta)\hspace{0.1em}a^+(x)$, for all $\theta\in(0,1)$ and all $x\in\mathrm{supp}(a^+)$ (hence for all $x\in\R^N$). Letting $\theta\to0^+$, we obtain
$$ \liminf_{\eps\to0^+}\hspace{0.1em}u_\eps(x)\geq a^+(x) \text{ for all }x\in\R^N. $$
This thereby completes the proof of Theorem~\ref{THEO}.

\subsection{Proof of Theorem~\ref{THEO02}}

We now complete this section by proving Theorem~\ref{THEO02}.
By assumption, there exists some $\beta>N$ such that $J$ has a finite $\beta$-th order moment. Then, by Lemma~\ref{LE:SUPER}, there exists a super-solution $\overline{u}_{\eps,\beta}\in C_0(\R^N)$ to \eqref{KPP} satisfying
$$ \overline{u}_{\eps,\beta}(x)\leq \frac{C_{\eps,\beta}}{1+|x|^\beta}\text{ for all }x\in\R^N, $$
and some $C_{\eps,\beta}>0$. In particular, $\overline{u}_{\eps,\beta}\in L^1(\R^N)$. Using the same arguments as above (relying on the comparison principle Lemma~\ref{LE:COMPARISON:BR}), we obtain that the unique positive minimal solution $u_\eps$ to \eqref{KPP} satisfies $u_\eps\leq\overline{u}_{\eps,\beta}$ in $\R^N$, which enforces that $u_\eps\in L^1(\R^N)$.
\smallskip

Let us now prove that the minimal solution $u_\eps$ is the only nontrivial nonnegative solution to \eqref{KPP}. Suppose that there is another nontrivial nonnegative solution $v_\eps\in L^\infty(\R^N)$. By Proposition~\ref{LE:REG:POS}, we have $v_\eps\in C_0(\R^N)$ and $v_\eps>0$ in $\R^N$. Since $u_\eps$ and $v_\eps$ are two solutions to \eqref{KPP} and since $u_\eps\in L^1(\R^N)$, we have
$$ \int_{\R^N}\left[u_\eps(x)\hspace{0.1em}\big(\mathcal{M}_{\eps,m}[v_\eps](x)+f(x,v_\eps(x))\big)-v_\eps(x)\hspace{0.1em}\big(\mathcal{M}_{\eps,m}[u_\eps](x)+f(x,u_\eps(x))\big)\right]\mathrm{d}x=0.
$$
In turn, this equality implies that
\begin{align}
\int_{\R^N}u_\eps(x)v_\eps(x)(v_\eps(x)-u_\eps(x))\hspace{0.1em}\mathrm{d}x=0. \label{unicite:sol:glob}
\end{align}
By minimality of $u_\eps$, we must have $u_\eps\leq v_\eps$ in $\R^N$. This, together with \eqref{unicite:sol:glob}, enforces that $u_\eps\equiv v_\eps$. Therefore, \eqref{KPP} indeed admits a unique positive solution.
\smallskip

It remains to derive the asymptotic behavior of $u_\eps$. Suppose that $a^+\in\mathring{B}_{1,\infty}^{m}(\R^N)$. Define $v_\eps:=a^+-u_\eps$ and $w_\eps:=a-u_\eps$. Then,
\begin{align*}
\frac{1}{\eps^m}(v_\eps-J_\eps\ast v_\eps)+u_\eps w_\eps=\frac{1}{\eps^m}(a^+-J_\eps\ast a^+). 
\end{align*}
Multiplying this equation by $v_\eps^+$ and integrating over $\R^N$, we get
\begin{align*}
\int_{\R^N}u_\eps(w_\eps^+)^2=\int_{\R^N}u_\eps\hspace{0.1em} v_\eps^+ w_\eps\leq \frac{1}{\eps^m}\int_{\R^N}v_\eps^+&(a^+-J_\eps\ast a^+)\leq\frac{1}{\eps^m}\int_{\R^N}a^+|a^+-J_\eps\ast a^+|, 
\end{align*}
where we have used \cite[Formula (5.7)]{BCV} and the fact that
\begin{align}
0\leq w_\eps^+= v_\eps^+\leq a^+\in L^1(\R^N). \label{SUPP:We}
\end{align}
It follows that
\begin{align*}
\int_{\R^N}u_\eps(w_\eps^+)^2&\leq \|a^+\|_\infty\hspace{0.1em}\frac{M_m(J)}{2}\int_{\R^N}\int_{\R^N}\rho_\eps(y)\hspace{0.1em}\frac{|\Delta_y^2a^+(x)|}{|y|^m}\hspace{0.1em}\mathrm{d}x\hspace{0.05em}\mathrm{d}y,
\end{align*}
where $\rho_\eps$ is as in \eqref{mollifier} and $\Delta_y^2a^+(x)=a^+(x+y)-2\hspace{0.05em}a^+(x)+a^+(x-y)$.
Since $a^+\in \mathring{B}_{1,\infty}^{m}(\mathbb{R}^N)$, using \cite[Proposition 6.1]{JB}, we get
$$ \limsup_{\eps\to0^+}\int_{\R^N}u_\eps(w_\eps^+)^2\leq \|a^+\|_\infty\hspace{0.1em}\frac{M_m(J)}{2}\hspace{0.1em}\lim_{|h|\to0}\hspace{0.1em}\frac{\|\Delta_h^2a^+\|_{L^1(\R^N)}}{|h|^m}=0. $$
And so we have
\begin{align}
\lim_{\eps\to0^+}\int_{\R^N}u_\eps(w_\eps^+)^2=0. \label{limitzero}
\end{align}
On the other hand, since $u_\eps\in L^1(\R^N)$, we can integrate \eqref{KPP} and we get
$$ \int_{\R^N}u_\eps(a-u_\eps)^+=\int_{\R^N}u_\eps(a-u_\eps)^-. $$
Thus, recalling \eqref{SUPP:We}, we have
\begin{align*}
\int_{\R^N}u_\eps(a-u_\eps)^\pm&=\int_{\mathcal{A}}u_\eps(a-u_\eps)^+\leq \|u_\eps\|_{L^1(\mathcal{A})}^{1/2}\left(\int_{\R^N}u_\eps[(a-u_\eps)^+]^2\right)^{1/2},
\end{align*}
where $\mathcal{A}:=\mathrm{supp}(a^+)$. Since $u_\eps\leq\|a^+\|_\infty$, it follows from \eqref{limitzero} that $\|u_\eps(a-u_\eps)\|_{L^1(\R^N)}$ vanishes as $\eps\to0^+$.
Therefore, up to extraction of a subsequence, we have
$$ \lim_{\eps\to0^+}\hspace{0.1em}u_\eps(x)(a(x)-u_\eps(x))=0\text{ for all }x\in\R^N\setminus\mathscr{N}, $$
where $\mathscr{N}\subset\R^N$ is a null set.
But since $u_\eps>0$, we then have
\begin{align*}
0\leq \lim_{\eps\to0^+}\hspace{0.1em}u_\eps^2(x)\leq \lim_{\eps\to0^+}\hspace{0.1em}|u_\eps(x)(a(x)-u_\eps(x))|=0 \text{ for all }x\in \R^N\setminus(\mathrm{supp}(a^+)\cup\mathscr{N}).
\end{align*}
Therefore, $u_\eps\to0$ in $\R^N\setminus(\mathrm{supp}(a^+)\cup\mathscr{N})$. By Lemma~\ref{LE:SUBSOL}, we have
$$ \forall\,z\in\mathrm{supp}(a^+),\ \exists\,V_z\in\mathscr{V}_z(\mathrm{supp}(a^+)),\ \exists\,\tau_z,\eps_z>0\ :\ \forall\, x\in V_z,\ \forall\,\eps\in(0,\eps_z),\ \tau_z\leq u_\eps(x), $$
where $\mathscr{V}_z(\mathrm{supp}(a^+))$ is the set of all neighborhoods of $z$ contained in $\mathrm{supp}(a^+)$.
Hence,
\begin{align*}
0\leq\tau_z\hspace{0.1em}\lim_{\eps\to0^+}\hspace{0.1em}|a^+(x)-u_\eps(x)|&\leq\lim_{\eps\to0^+}\hspace{0.1em}|u_\eps(x)(a(x)-u_\eps(x))|=0 \text{ for all }x\in V_z\setminus\mathscr{N}.
\end{align*}
This being true for any $z\in\mathrm{supp}(a^+)$, it follows that $u_\eps\to a^+$ a.e. in $\R^N$, as desired.


\appendix

\section{}

In this Appendix, we state a last result which suggests that the moment condition in Theorem~\ref{THEO} (and in Lemma~\ref{LE:SUBSOL}) is optimal. In substance, we prove that, if the $\beta$-th order moment of $J$ is infinite for some $\beta<m$, then there cannot exist uniform (with respect to $\eps$) sub-solutions $\underline{u}$ to \eqref{KPP} with $\underline{u}\leq a^+$, i.e. that Lemma~\ref{LE:SUBSOL} cannot be significantly improved.
\begin{prop}\label{MOMENT}
Let $0<m<2$. Assume \eqref{HYP:A1} and that $J\in L^1(\R^N)$ is a radially symmetric kernel with unit mass. Suppose that there exists $\eps_0>0$ and $\underline{u}\in C_c(\R^N)$ such that $\underline{u}$ is a nonnegative, nontrivial sub-solution to \eqref{KPP} for all $0<\eps<\eps_0$. Then,
$$ \limsup_{\beta\to m^-}\,\,(m-\beta)\int_{\R^N}J(x)|x|^{\beta}\mathrm{d}x<\infty. $$
In particular, $J$ has a finite $\beta$-th order moment for all $0<\beta<m$.
\end{prop}
\begin{proof}
Let $\eps_0>0$ and $\underline{u}\in C_c(\R^N)$ be such that $\underline{u}$ is a nonnegative sub-solution to \eqref{KPP} for all $0<\eps<\eps_0$. Assume, without loss of generality, that $\eps_0=1$ and that $\mathrm{supp}(\underline{u})=B_1$. Then, by assumption, we have $\eps^{-m}(J_\eps\ast\underline{u}-\underline{u})+\underline{u}(a-\underline{u})\geq0\text{ in }\R^N$.
Multiplying this by $\underline{u}$ and integrating over $\R^N$, we obtain
$$ \frac{1}{2\hspace{0.05em}\eps^m}\int_{\R^N}\int_{\R^N}J_\eps(y)(\underline{u}(x-y)-\underline{u}(x))^2\mathrm{d}x\hspace{0.05em}\mathrm{d}y\leq \int_{\R^N}\underline{u}^2a^+=:I. $$
But since $\underline{u}$ is supported in $B_1$, we have
\begin{align*}
I&\geq\frac{1}{2\hspace{0.05em}\eps^m}\int_{\R^N}\int_{|x|\geq1}J(y)\hspace{0.1em}\underline{u}(x-\eps\hspace{0.05em}y)^2\mathrm{d}x\hspace{0.05em}\mathrm{d}y=\frac{1}{2\hspace{0.05em}\eps^m}\int_{\R^N}\int_{|z+\eps y|\geq1}J(y)\hspace{0.1em}\underline{u}(z)^2\mathrm{d}z\hspace{0.05em}\mathrm{d}y.
\end{align*}
Dividing both sides of this inequality by $\eps^{\ell}$ for some $\ell\in(0,1)$ and integrating with respect to $\eps\in(0,1)$, we obtain
\begin{align*}
(1-\ell)^{-1}I&\geq \frac{1}{2}\int_0^1\left(\int_{\R^N}\int_{|z+\eps y|\geq 1}J(y)\hspace{0.1em}\underline{u}(z)^2\hspace{0.1em}\mathrm{d}y\hspace{0.05em}\mathrm{d}z\right)\frac{\mathrm{d}\eps}{\eps^{m+\ell}} \\
&=\frac{1}{2}\int_{\R^N}\int_0^{|y|}\int_{|z+\delta y/|y||\geq1}J(y)\hspace{0.1em}\underline{u}(z)^2\hspace{0.1em}|y|^{m+\ell-1}\hspace{0.1em}\frac{\mathrm{d}\delta}{\delta^{m+\ell}}\mathrm{d}y\hspace{0.05em}\mathrm{d}z
\end{align*}
Observe that if $y\in\R^N\setminus B_3$, $\delta\in[2,3]$ and $z\in B_1$, then we have $0<\delta\leq|y|$ and $|z+\delta\hspace{0.1em}y/|y||\geq1$. We thus obtain the lower bound
\begin{align*}
(1-\ell)^{-1}I&\geq \frac{1}{2}\int_2^3\frac{\mathrm{d}\delta}{\delta^{m+\ell}}\times\int_{\R^N\setminus B_3}J(y)|y|^{m+\ell-1}\mathrm{d}y\times\int_{B_1}\underline{u}(z)^2\mathrm{d}z,
\end{align*}
which thereby completes the proof.
\end{proof}
\begin{remark}
Interestingly, the proof of Proposition~\ref{MOMENT} is similar to that of \cite[Proposition 3]{LM}. In \cite{LM}, Lamy and Mironescu show that some Besov spaces can be characterized by quantities which strikingly resemble the nonlocal diffusion part of \eqref{KPP}, namely
\begin{align}
\sup_{0<\eps<1}\hspace{0.1em}\frac{\|J_\eps\ast f-f\|_{L^p(\R^N)}}{\eps^s}, \label{LaMi}
\end{align}
provided $J$ has a finite $s$-th order moment.
We believe that a fine understanding of \eqref{LaMi} in the case where $f$ also depends on $\eps$ could be helpful to study the influence of tail of the kernel on the asymptotic properties of solutions to \eqref{KPP}.
\end{remark}

\section*{Acknowledgments} The author warmly thanks \emph{J\'er\^ome Coville} who looked through a preliminary version of the present manuscript and offered many comments. The author also thanks the anonymous referee whose insightful comments helped improve and clarify the manuscript. This project has been supported by the French National Research Agency (ANR) in the framework of the ANR NONLOCAL project (ANR-14-CE25-0013).

\vspace{2mm}


\begin{thebibliography}{12}
\setlinespacing{0.95}
\frenchspacing

\bibitem{Allen} {\sc A. P. Allen, B.-L. Li, E. L. Charnov}:
Population fluctuations, power laws and mixtures of lognormal distributions,
{\em  Ecol. Lett.}, 4, (2001), pp. 1–3.


\bibitem{Bartumeus} {\sc F. Bartumeus}:
L\'evy processes in animal movement: an evolutionary hypothesis,
{\em Fractals}, 15(2), (2007), pp. 151-162.

\bibitem{Bartumeus2} {\sc F. Bartumeus}:
Behavioral intermittence, L\'evy patterns, and randomness in animal movement,
{\em Oikos}, 118(4), (2009), pp. 488-494.

\bibitem{Bates} {\sc P. W. Bates, P. C. Fife, X. Ren, X. Wang}:
Travelling  waves in a convolution model for phase transitions,
{\em Arch. Ration. Mech. Anal.}, 138, (1997), pp. 105-136.

\bibitem{BaZh} {\sc P. W. Bates, G. Zhao}:
Existence, uniqueness and stability of the stationary solution to a nonlocal evolution equation arising in population dispersal,
{\em J. Math. Anal. Appl.}, 332(1), (2007), pp. 428-440.


\bibitem{BCV} {\sc H. Berestycki, J. Coville, H.-H. Vo}:
Persistence criteria for populations with non-local dispersion,
{\em J. Math. Biol.}, 72, (2016), pp. 1693-1745.

\bibitem{Bouin} {\sc E. Bouin, J. Garnier, C. Henderson, F. Patout}:
Thin front limit of an integro-differential Fisher-KPP equation with fat-tailed kernels,
{\em SIAM J. Math. Anal.}, 50(3), (2018), pp. 3365–3394.

\bibitem{BBM} {\sc J. Bourgain, H. Brezis, P. Mironescu}:
Another look at Sobolev spaces,
{\em Optimal Control and Partial Differential Equations} (J.L. Menaldi, E. Rofman and A. Sulem, eds.) A volume in honour of A. Bensoussan's 60th birthday, IOS Press., (2001), pp. 439-455.

\bibitem{JB} {\sc J. Brasseur}:
A Bourgain-Brezis-Mironescu characterization of higher order Besov-Nikol'skii spaces,
{\em Ann. Inst. Fourier}, 68(4), (2018), pp. 1671-1714.

\bibitem{BCHV} {\sc J. Brasseur, J. Coville, F. Hamel, E. Valdinoci}, Liouville type results for a nonlocal obstacle problem, {\em PLMS}, 119(2), (2019), pp. 291-328.

\bibitem{BC} {\sc J. Brasseur, J. Coville}:
A counterexample to the Liouville property of some nonlocal problems,
{\em Ann. IHP Anal. Non Linéaire (to appear)}, (2018)

\bibitem{Cain} {\sc M. L. Cain, B. G. Milligan, A. E. Strand}:
Long-distance seed dispersal in plant populations,
{\em Am. J. Bot.}, 87(9), (2000), pp. 1217–1227.

\bibitem{Cantrell} {\sc R. S. Cantrell, C. Cosner, K.-Y. Lam}:
Resident-invader dynamics in infinite dimensional systems,
{\em J. Diff. Eq.}, 263(8), (2017), pp. 4565–4616.

\bibitem{Chapman} {\sc D. S. Chapman, C. Dytham, G. S. Oxford}:
Modelling population redistribution in a leaf beetle: an evaluation of alternative dispersal functions,
{\em J. Anim. Ecol.}, 76(1), (2007), pp. 36-44.

\bibitem{Clark} {\sc J. S. Clark}:
Why trees migrate so fast: confronting theory with dispersal biology and the paleorecord,
{\em The American Naturalist}, 152(2), (1998), pp. 204-224.

\bibitem{Clark2} {\sc J. S. Clark, C. Fastie, G. Hurtt, S. T. Jackson, C. Johnson, G. A. King, M. Lewis, J. Lynch, S. Pacala, C. Prentice, E. Schupp, T. Webb III, P. Wyckoff}:
Reid's paradox of rapid plant migration,
{\em Bio-Science}, 48, (1998), pp. 13–24.

\bibitem{CovDup07} {\sc J. Coville, L. Dupaigne}:
On a non-local reaction diffusion equation arising in population dynamics,
{\em Proc. Math. Roy. Soc. of Edinburgh Sect. A}, 137(4), (2007), pp. 1-29.

\bibitem{Cov07} {\sc J. Coville}:
Travelling fronts in asymmetric nonlocal reaction diffusion equations: the bistable and ignition cases,
{\em Preprint: hal-00696208}, (2007).

\bibitem{Cov08} {\sc J. Coville}:
Remarks on the strong maximum principle for nonlocal operators,
{\em Electron. J. Diff. Eqns.}, 66, (2008), pp. 1-10.

\bibitem{Cov10} {\sc J. Coville}:
On a simple criterion for the existence of a principal eigenfunction of some nonlocal oper-ators,
{\em J. Diff. Eq.}, 249(11), (2010), pp. 2921-2953.



\bibitem{Garnier} {\sc J. Garnier}:
Accelerating solutions in integro-differential equations,
{\em  SIAM J. Math. Anal.}, 43(4), (2011), pp. 1955-1974.

\bibitem{Gilbarg} {\sc D. Gilbarg, N. S. Trudinger}:
Elliptic partial differential equations of second order,
{\em Springer-Verlag Berlin Heidelberg}, (1977).

\bibitem{Hallat} {\sc O. Hallatschek, D. S. Fisher}:
Acceleration of evolutionary spread by long-range dispersal,
{\em PNAS}, 111(46), (2014), pp. E4911-E4919.

\bibitem{Humphries} {\sc N. E. Humphries, N. Queiroz, J. R. M. Dyer, N. G. Pade, M. K. Mu-syl, K. M. Schaefer, D. W. Fuller, J. M. Brunnschweiler, T. K.Doyle, J. D. R. Houghton, G. C. Hays, C. S. Jones, L. R. Noble, V. J. Wearmouth, E. J. Southall, D. W. Sims}:
Environmental context  explains  L\'evy  and  Brownian  movement  patterns  of  marine  predators,
{\em Nature}, 465, (2010), pp. 1066-1069.

\bibitem{Hutson} {\sc V. Hutson, S. Martinez, K. Mischaikow, G. T. Vickers}:
The evolution of dispersal,
{\em J. Math. Biol.}, 47(6), (2003), pp. 483–517.

\bibitem{Kao} {\sc C.-Y. Kao, Y. Lou, W. Shen}:
Random dispersal vs. non-local dispersal,
{\em  Discrete Contin. Dyn. Syst.}, 26(2), (2010), pp. 551–596.

\bibitem{LM} {\sc X. Lamy, P. Mironescu}:
Characterization of function spaces via low regularity mollifiers,
{\em Discrete Contin. Dyn. Syst.}, 35(12), (2015), pp. 6015-6030.

\bibitem{LI} {\sc W.-T. Li, Y.-H. Su, F.-Y. Yang}:
Asymptotic behaviors for nonlocal diffusion equations about the dispersal spread,
{\em Preprint: arXiv:1911.07665v1}, (2019).


\bibitem{Petrovskii} {\sc S. Petrovskii, A. Morozov, B.-L. Li}:
On a possible origin of the fat-tailed dispersal in population dynamics,
{\em Ecol. Complex.}, 5(2), (2008), pp. 146-150.

\bibitem{Ponce} {\sc A. Ponce}:
An estimate in the spirit of Poincar\'e's inequality,
{\em J. Eur. Math. Soc.}, 6, (2004), pp. 1-15.

\bibitem{Radinger} {\sc J. Radinger, C. Wolter}:
Patterns and predictors of fish dispersal in rivers,
{\em Fish and Fisheries}, 15(3), (2013), pp. 456-473.

\bibitem{Schurr} {\sc F. M. Schurr, O. Steinitz, R. Nathan}:
Plant fecundity and seed dispersal in spatially heterogeneous environments: models, mechanisms and estimation,
{\em J. Ecol.}, 96(4), (2008), pp. 628–641.

\bibitem{VoShen2} {\sc Z. Shen, H.-H. Vo}:
Nonlocal dispersal equations in time-periodic media: principal spectral theory, limiting properties and long-time dynamics,
{\em J. Diff. Eq.}, 267(2), (2019), pp. 1423-1466.

\bibitem{Vo} {\sc H.-H. Vo}:
Principal spectral theory of time-periodic nonlocal dispersal operators of Neumann type,
{\em Preprint: arXiv:1911.06119}, (2019).

\bibitem{Yagisita} {\sc H. Yagisita}:
Existence and nonexistence of traveling waves for a nonlocal monostable equation,
{\em Publ. RIMS, Kyoto Univ.}, 45, (2009), pp. 925–953.

\end{thebibliography}
\end{document}